\documentclass[secthm,seceqn,amsthm,ussrhead,12pt]{amsart}
\usepackage[utf8]{inputenc}
\usepackage[english]{babel}
\usepackage{amssymb,amsmath,amsthm,amsfonts,xcolor,enumerate,hyperref,comment,longtable,cleveref}

\usepackage{times}
\usepackage{cite}
\usepackage{pdflscape}
\usepackage{ulem}
\usepackage[mathcal]{euscript}
\usepackage{tikz}
\usepackage{hyperref}
\usepackage{cancel}
\usepackage{stmaryrd}

\usetikzlibrary{arrows}

\newtheorem{Th}{Theorem}
\newtheorem{Lem}[Th]{Lemma}

\newtheorem{corollary}[Th]{Corollary}

\newenvironment{Proof}[1][Proof.]{\begin{trivlist}
\item[\hskip \labelsep {\bfseries #1}]}{\flushright
$\Box$\end{trivlist}}

\usepackage{stmaryrd}
\usepackage{xcolor}
\newcommand{\red}{\color{black}}

\begin{document}
	\sloppy

{\Large The geometric classification of Leibniz algebras
\footnote{The work was supported by 
RFBR 17-01-00258 
and by the President's Program "Support of Young Russian Scientists" (grant MK-1378.2017.1).}}

\medskip

\medskip

\medskip

\medskip
\textbf{Nurlan Ismailov$^{a}$, Ivan Kaygorodov$^{b}$, Yury Volkov$^{c}$}
\medskip

{\tiny

$^{a}$ Suleyman Demirel University, Almaty, Kazakhstan 

$^{b}$ Universidade Federal do ABC, CMCC, Santo Andr\'{e}, Brazil.

$^{c}$ Saint Petersburg State University, Saint Petersburg, Russia.
\smallskip

    E-mail addresses:\smallskip
    
    Nurlan Ismailov (nurlan.ismail@gmail.com),
    
    Ivan Kaygorodov (kaygorodov.ivan@gmail.com),
    
    Yury Volkov (wolf86\_666@list.ru).

}

       \vspace{0.3cm}

{\bf Abstract.} 
We describe all rigid algebras and all irreducible components in the variety of four dimensional Leibniz algebras $\mathfrak{Leib}_4$ over $\mathbb{C}.$  
In particular, we prove that the Grunewald--O'Halloran conjecture is not valid and the Vergne conjecture is valid for $\mathfrak{Leib}_4.$
\smallskip

{\bf Keywords:} Leibniz algebra, Grunewald--O'Halloran conjecture, Vergne conjecture, orbit closure, degeneration, rigid algebra
       \vspace{0.3cm}

       \vspace{0.3cm}

{\it 2010 MSC}: 14D06, 14L30.

\section{Introduction}


Degenerations of algebras is an interesting subject, which was studied in various papers 
(see, for example, \cite{M79,CKLO13,BC99,S90,GRH,GRH2,BB14,kpv16,kppv,kv17}).
In particular, there are many results concerning degenerations of algebras of low dimensions in a variety defined by a set of identities.
One of important problems in this direction is the description of so-called rigid algebras. 
These algebras are of big interest, since the closures of their orbits under the action of generalized linear group form irreducible components of a variety under consideration (with respect to the Zariski topology). 
For example, the rigid algebras were classified in the varieties of
two dimensional bicommutative algebras in \cite{kv17},  
three dimensional Novikov algebras in \cite{BB14},
three dimensional Leibniz in \cite{ra12}, 
four dimensional Lie algebras in \cite{BC99}, 
four dimensional Jordan algebras in \cite{KE14},
four dimensional Zinbiel algebras and nilpotent four dimensional  Leibniz algebras in \cite{kppv},
unital five dimensional associative algebras in \cite{M79}, 
nilpotent five- and six-dimensional Lie algebras in \cite{S90,GRH}, 
nilpotent five- and six-dimensional Malcev algebras in \cite{kpv16},
and some other.

The Leibniz algebras were introduced as a generalization of Lie algebras.
The study of the structure theory and other properties of Leibniz algebras was initiated by Loday in \cite{lodaypir}.
Leibniz algebras were also studied in \cite{leib2,leib4,yau}.
An algebra $A$ is called a {\it  Leibniz  algebra}  if it satisfies the identity
$$(xy)z=(xz)y+x(yz).$$ 
It is easy to see that any Lie algebra is a Leibniz algebra.
At this moment, the algebraic classification of $n$ dimensional Leibniz algebras over $\mathbb{C}$ is known only for $n\le 4$. 

In this paper we describe all rigid algebras and all irreducible components in variety of $\mathfrak{Leib}_4$.
As a result, we show that  $\mathfrak{Leib}_4$ has 
{\red $5$ rigid algebras and $16$ irreducible components. }

\section{Definitions and notation}

All spaces in this paper are considered over $\mathbb{C}$, and we write simply $dim$, $Hom$ and $\otimes$ instead of $dim_{\mathbb{C}}$, $Hom_{\mathbb{C}}$ and $\otimes_{\mathbb{C}}$. An algebra $A$ is a set with a structure of a vector space and a binary operation that induces a bilinear map from $A\times A$ to $A$.

Given an $n$-dimensional vector space $V$, the set $Hom(V \otimes V,V) \cong V^* \otimes V^* \otimes V$ 
is a vector space of dimension $n^3$. This space has a structure of the affine variety $\mathbb{C}^{n^3}.$ Indeed, let us fix a basis $e_1,\dots,e_n$ of $V$. Then any $\mu\in Hom(V \otimes V,V)$ is determined by $n^3$ structure constants $c_{i,j}^k\in\mathbb{C}$ such that
$\mu(e_i\otimes e_j)=\sum\limits_{k=1}^nc_{i,j}^ke_k$. A subset of $Hom(V \otimes V,V)$ is {\it Zariski-closed} if it can be defined by a set of polynomial equations in the variables $c_{i,j}^k$ ($1\le i,j,k\le n$).

All algebra structures on $V$ satisfying Leibniz identity form a Zariski-closed subset of the variety $Hom(V \otimes V,V)$. 
We denote this subset by $\mathfrak{Leib}_n$.
The general linear group $GL(V)$ acts on $\mathfrak{Leib}_n$ by conjugations:
$$ (g * \mu )(x\otimes y) = g\mu(g^{-1}x\otimes g^{-1}y)$$ 
for $x,y\in V$, $\mu\in \mathfrak{Leib}_n\subset Hom(V \otimes V,V)$ and $g\in GL(V)$.
Thus, $\mathfrak{Leib}_n$ is decomposed into $GL(V)$-orbits that correspond to the isomorphism classes of algebras. 
Let $O(\mu)$ denote the orbit of $\mu\in \mathfrak{Leib}_n$ under the action of $GL(V)$ and $\overline{O(\mu)}$ denote the Zariski closure of $O(\mu)$.

Let $A$ and $B$ be two $n$-dimensional Leibniz algebras and $\mu,\lambda \in \mathfrak{Leib}_n$ represent $A$ and $B$ respectively.
We say that $A$ degenerates to $B$ and write $A\to B$ if $\lambda\in\overline{O(\mu)}$.
Note that in this case we have $\overline{O(\lambda)}\subset\overline{O(\mu)}$. 
Hence, the definition of a degeneration does not depend on the choice of $\mu$ and $\lambda$. 
 We write $A\not\to B$ if $\lambda\not\in\overline{O(\mu)}$.

Let $A$ be represented by $\mu \in \mathfrak{Leib}_n$. 
Then  $A$ is  {\it rigid} in $\mathfrak{Leib}_n$ if $O(\mu)$ is an open subset of $\mathfrak{Leib}_n.$ 
 Recall that a subset of a variety is called irreducible if it cannot be represented as a union of two non-trivial closed subsets. 
 A maximal irreducible closed subset of a variety is called {\it irreducible component}.
 It is well known that any affine variety can be represented as a finite union of its irreducible components in a unique way.
 In particular, $A$ is rigid in $\mathfrak{Leib}_n$  iff $\overline{O(\mu)}$ is an irreducible component of $\mathfrak{Leib}_n.$
 This is a general fact about algebraic varieties whose proof can be found, for example, in \cite{kppv}. 
 
We use the following notation: 
\begin{enumerate}
\item $Ann_L(A)=\{ a \in A \mid xa =0 \mbox{ for all } x\in A \}$ is the left  annihilator of $A;$
\item $Ann_R(A)=\{ a \in A \mid ax =0 \mbox{ for all } x\in A \}$ is the right  annihilator of $A;$
\item $Ann(A)=Ann_R(A)\cap Ann_L(A)$ is the  annihilator of $A;$
\item $A^{(+2)}$ is the space $\{xy+yx\mid x,y\in A\}.$
\end{enumerate}
Given spaces $U$ and $W$, we write simply $U>W$ instead of $dim\,U>dim\,W$. 
We write $UW$ ($U,W\subset V$) for the product of subspaces of $V$ with respect to the multiplication $\mu$. 
We use the notation $S_i=\langle e_i,\dots,e_4\rangle,\ i=1,\ldots,4$.

\section{Methods} 

In the present work we use the methods applied in our previous works (see \cite{kv17, kpv16,kppv}).

To prove degenerations, we will construct families of matrices parametrized by $t$. Namely, let $A$ and $B$ be two Leibniz algebras represented by the structures $\mu$ and $\lambda$ from $\mathfrak{Leib}_n$ respectively. 
Let $e_1,\dots, e_n$ be a basis of $V$ and $c_{i,j}^k$ ($1\le i,j,k\le n$) be the structure constants of $\lambda$ in this basis. If there exist $a_i^j(t)\in\mathbb{C}$ ($1\le i,j\le n$, $t\in\mathbb{C}^*$) such that $E_i^t=\sum\limits_{j=1}^na_i^j(t)e_j$ ($1\le i\le n$) form a basis of $V$ for any $t\in\mathbb{C}^*$, and the structure constants of $\mu$ in the basis $E_1^t,\dots, E_n^t$ are such polynomials $c_{i,j}^k(t)\in\mathbb{C}[t]$ that $c_{i,j}^k(0)=c_{i,j}^k$, then $A\to B$. In this case  $E_1^t,\dots, E_n^t$ is called a {\it parametrized basis} for $A\to B$.
It is easy to see that any algebra degenerates to the algebra with zero multiplication. 

Let now $A(*):=\{A(\alpha)\}_{\alpha\in I}$ be a set of algebras, and let $B$ be another algebra. Suppose that, for $\alpha\in I$, $A(\alpha)$ is represented by the structure $\mu(\alpha)\in\mathfrak{Leib}_n$ and $B\in\mathfrak{Leib}_n$ is represented by the structure $\lambda$. Then $A(*)\to B$ means $\lambda\in\overline{\{O(\mu(\alpha))\}_{\alpha\in I}}$, and $A(*)\not\to B$ means $\lambda\not\in\overline{\{O(\mu(\alpha))\}_{\alpha\in I}}$.

Let $A(*)$, $B$, $\mu(\alpha)$, and $\lambda$ be as above. To prove $A(*)\to B$ it is enough to construct a family of pairs $(g(t), f(t))$ parametrized by $t\in\mathbb{C}^*$, where $f(t)\in I$ and $g(t)\in GL(V)$,
with the following properties: Let $e_1,\dots, e_n$ be a basis of $V$ and $c_{i,j}^k$ ($1\le i,j,k\le n$) be the structure constants of $\lambda$ in this basis. If we construct $a_i^j:\mathbb{C}^*\to \mathbb{C}$ ($1\le i,j\le n$) and $f: \mathbb{C}^* \to I$ such that $E_i^t=\sum\limits_{j=1}^na_i^j(t)e_j$ ($1\le i\le n$) form a basis of $V$ for any  $t\in\mathbb{C}^*$, and the structure constants of $\mu_{f(t)}$ in the basis $E_1^t,\dots, E_n^t$ are such polynomials $c_{i,j}^k(t)\in\mathbb{C}[t]$ that $c_{i,j}^k(0)=c_{i,j}^k$, then $A(*)\to B$. In this case  $E_1^t,\dots, E_n^t$ and $f(t)$ are called a parametrized basis and a {\it parametrized index} for $A(*)\to B$ respectively.

Note also the following fact. Suppose that, for $\alpha\in\mathbb{C}$, the structure $\mu(\alpha)\in \mathfrak{Leib}_n$ has structure constants $c_{i,j}^k(\alpha)\in\mathbb{C}$ in the basis $e_1,\dots,e_n$, where $c_{i,j}^k(t)\in\mathbb{C}[t]$ for all $1\le i,j,k\le n$. Let $X$ be some subset of $\mathfrak{Leib}_n$ such that
$\mu(\alpha)\in \overline{X}$ for $\alpha\in\mathbb{C}\setminus S$, where $S$ is a finite subset of $\mathbb{C}$. Then $\mu(\alpha)\in \overline{X}$ for all $\alpha\in\mathbb{C}$. 
Indeed, $\mu(\alpha)\in \overline{ \{\mu(\beta) \}_{\beta\in\mathbb{C}\setminus S}}\subset \overline{X}$  for any $\alpha\in\mathbb{C}$. Thus, to prove that $\mu(\alpha)\in \overline{X}$ for all $\alpha\in\mathbb{C}$ we will prove that $\mu(\alpha)\in \overline{X}$ for all but finitely many $\alpha$.

Note that $A(*)\not\to B$ if $dim\,Der(A(\alpha))>dim\,Der(B)$ for all $\alpha\in I$, where $Der(A)$ is the Lie algebra of derivations of $A$. Moreover, in the case of one algebra, if $A\to B$ and $A\not\cong B$, then $Der(A)<Der(B)$.

In other cases, the main tool for proving assertions of the form $A(*)\not\to B$ will be the following lemma.

\begin{Lem}[Lemma 2 of \cite{kppv}]\label{gmain}
Let $\mathcal{B}$ be a Borel subgroup of $GL(V)$ and $\mathcal{R}\subset \mathfrak{Leib}_n$ be a $\mathcal{B}$-stable closed subset.
If $A(*) \to B$ and for any $\alpha\in I$ the algebra $A(\alpha)$ can be represented by a structure $\mu(\alpha)\in\mathcal{R}$, then there is $\lambda\in \mathcal{R}$ representing $B$.
\end{Lem}

In all our applications of Lemma \ref{gmain} we will use the group of lower triangular matrices as a Borel subgroup.
This lemma can be applied in the case $|I|=1$ to prove non-degenerations. In particular, if $Ann_L(A)>Ann_L(B)$, $A^2<B^2$ or $A^{(+2)}<B^{(+2)}$, then $A\not\to B$ by Lemma \ref{gmain}. We will use also the following two criteria that follow from Lemma \ref{gmain}. If $A\to B$, then the dimension of a maximal trivial subalgebra of $B$ is greater or equal to the dimension of a maximal trivial subalgebra of $A$ and the dimension of a maximal anticommutative subalgebra $D$ of $B$ such that $dim(BD)\le 1$ is greater or equal to the dimension of a maximal anticommutative subalgebra of $A$ satisfying the same condition.

In some cases we will construct the set $\mathcal{R}$ for the assertion $A(*)\not\to B$ explicitly. In this case we will define $\mathcal{R}$ as a set of structures $\mu$ satisfying some polynomial equations. In such a description we always denote by $c_{ij}^k$ ($1\le i,j,k\le n$) the structure constants of $\mu$. Note also that in this case a condition of the form $S_iS_j\subset S_k$ is equivalent to a set of polynomial equations. Moreover, the set defined by such a condition is stable under the action of the group of lower triangular matrices.

\section{Main theorem}

The goal of this section is to describe irreducible components in $\mathfrak{Leib}_4$.
The algebraic classification of four dimensional Lebniz algebras  is based on the papers \cite{ort13,  Abror13, alb06}.
Let us give a brief introduction to this classification.
The analogue of Levi-Malcev's theorem about the splitting of the solvable radical for Leibniz algebras was proved in \cite{barnes}.
Note that any semisimple (here this means that the solvable radical is zero) Leibniz algebra is a Lie algebra, and 
hence a Leibniz algebra is always the hemisemidirect product
of a semisimple Lie algebra and a solvable ideal. The first step in the classification of four dimensional Leibniz algebras was done in \cite{alb06}, where all nilpotent four dimensional Leibniz algebras were classified. Then the description of all four dimensional solvable non-nilpotent Lebniz algebras was obtained in \cite{Abror13}.
Finally, it was proved in \cite{ort13} that  there is only one non-solvable indecomposable Leibniz algebra whose dimension is less or equal to four, namely, the simple Lie algebra $\mathfrak{sl}_2$.

Based on the just mentioned works, we have constructed Table 1 that describes all four dimensional non-nilpotent non-Lie Leibniz algebras. Let us also introduce the following four dimensional Lie algebra structures:
$$\begin{array}{llllll} 
\mathfrak{R}_2          & e_1e_2=-e_2e_1=e_2, & e_3e_4=-e_4e_3=e_4; \\
\mathfrak{sl}_2\oplus \mathbb C         & e_1e_2=-e_2e_1=e_2, & e_1e_3=-e_3e_1=-e_3, & e_2e_3=-e_3e_2=e_1; \\
g_4(a,b)                & e_1e_2=-e_2e_1=e_2, & e_1e_3=-e_3e_1=e_2+a e_3, & e_1e_4=-e_4e_1=e_3+b e_4; \\
g_5(a)                  & e_1e_2=-e_2e_1=e_2, & e_1e_3=-e_3e_1=e_2+a e_3, \\& e_1e_4=-e_4e_1=(a+1)e_4, &  e_2e_3=-e_3e_2=e_4.
\end{array}$$
Due to the results of \cite{BC99}, the variety of four dimensional Lie algebras contains $4$ irreducible components, namely, $\overline{O(\mathfrak{sl}_2\oplus \mathbb C)}$, $\overline{O(\mathfrak{R}_2)}$, $\overline{\{O(g_5(a))\}_{a \in\mathbb{C}}}$, and $\overline{\{O(g_4(a,b))\}_{a,b \in\mathbb{C}}}$.
In particular, there are $2$ rigid four dimensional Lie algebras, namely, $\mathfrak{sl}_2\oplus \mathbb C$ and $\mathfrak{R}_2$.

The main result of the present paper is the following theorem.

\begin{Th}\label{mainTh}
The variety $\mathfrak{Leib}_4$ has {\red $16$} irreducible components: 
$$\overline{\{O(g_4(a,b))\}_{a,b \in\mathbb{C}}}, 
\ \overline{\{O(g_5(a))\}_{a \in\mathbb{C}}},
\ \overline{\{O(\mathfrak{L}_4^a) \}_{a \in\mathbb{C}}},
\ \overline{\{O(\mathfrak{L}_8^a) \}_{a \in\mathbb{C}}},
\ \overline{\{O(\mathfrak{L}_9^a) \}_{a \in\mathbb{C}}},
\ \overline{\{O(\mathfrak{L}_{10}^a) \}_{a \in\mathbb{C}}},$$
$$\ \overline{\{O(\mathfrak{L}_{15}^a) \}_{a \in\mathbb{C}}},
\overline{\{O(\mathfrak{L}_{18}^a) \}_{a \in\mathbb{C}}},
\ \overline{\{O(\mathfrak{L}_{21}^{a,b}) \}_{a,b \in\mathbb{C}}},
\  \overline{\{O(\mathfrak{L}_{22}^{a,b}) \}_{a,b \in\mathbb{C}}},
\ \overline{\{O(\mathfrak{L}_{23}^{a,b}) \}_{a,b \in\mathbb{C}}},$$
$$\ \overline{O(\mathfrak{sl}_2\oplus \mathbb C)}, 
\ \overline{O(\mathfrak{R}_1)}, 
\ \overline{O(\mathfrak{R}_2)}, 
\ \overline{O(\mathfrak{R}_3)}, 
\ \overline{O(\mathfrak{L}_{44})}.
$$

In particular, there are {\red $5$} rigid four dimensional Leibniz algebras:
$$\mathfrak{sl}_2\oplus \mathbb C, \ \mathfrak{R}_1, \  \mathfrak{R}_2, \ \mathfrak{R}_3,   \ \mathfrak{L}_{44}.$$ 

\end{Th}

Non-Lie algebras mentioned in the theorem are described in Table 1. The rest of the paper is devoted to the proof of Theorem \ref{mainTh}.

\subsection{Nilpotent Leibniz algebras and conjectures about them.}
Several conjectures state that nilpotent Lie algebras form a very small subvariety in the variety of Lie algebras. Grunewald and O'Halloran conjectured in \cite{GRH3} that for any $n$-dimensional nilpotent Lie algebra $A$ there exists an $n$-dimensional non-nilpotent Lie algebra $B$ such that $B\to A$. At the same time,
Vergne conjectured in \cite{V70} that a nilpotent Lie algebra cannot be rigid in the variety of all Lie algebras. Analogous assertions can be conjectured for other varieties.
We will say that the variety $\mathfrak{A}$ of algebras has {\it Grunewald--O'Halloran Property} if for any nilpotent algebra $A\in\mathfrak{A}$ there is a non-nilpotent algebra $B\in\mathfrak{A}$ such that $B\to A$.
We will say that $\mathfrak{A}$ has {\it Vergne Property} if there are no nilpotent rigid algebras in $\mathfrak{A}$.
We will also say that $\mathfrak{A}$ has {\it Vergne--Grunewald--O'Halloran Property} if any irreducible component of $\mathfrak{A}$ contains a non-nilpotent algebra.
Grunewald--O'Halloran Property was proved for four dimensional Lie and Jordan algebras in \cite{BC99,KE14}
and for three dimensional Novikov and Leibniz algebras in \cite{BB14,CKLO13}. Also, some results concerning Grunewald--O'Halloran Conjecture for Lie algebras were obtained in \cite{ht14,ht16}.

It is clear that Vergne--Grunewald--O'Halloran Property follows from Grunewald--O'Halloran Property and Vergne Property follows from Vergne--Grunewald--O'Halloran Property.
As a part of the proof of Theorem \ref{mainTh}, we will show that $\mathfrak{Leib}_4$ has Vergne--Grunewald--O'Halloran Property. On the other hand, we will show that the same variety does not have Grunewald--O'Halloran Property.

Let us also introduce the following four dimensional nilpotent Lebniz  algebra structures:
$$\begin{array}{lcllllll} 
\mathfrak{N}_{3}^{a} & e_1e_1 = e_4,& e_1e_2 = a e_4,&  e_2e_1 = -a e_4,& e_2e_2 = e_4,&  e_3e_3 = e_4;  \\ 
\mathfrak{L}^n_2  & e_1e_1 = e_2, &e_2e_1 = e_3, & e_3e_1 = e_4;  \\ 
\mathfrak{L}^n_5 &  e_1e_1 = e_3, & e_2e_1 = e_3,& e_2e_2 = e_4, &e_3e_1=e_4; \\   
\mathfrak{L}^n_{11} & e_1e_1 = e_4,&  e_1e_2=-e_3,& e_1e_3=-e_4, &  e_2e_1 = e_3, &e_2e_2=e_4,&  e_3e_1=e_4.   \\  
\end{array}$$
It was proved in \cite{kppv} that the variety of four dimensional nilpotent Leibniz algebras is formed by four irreducible components, namely, $\overline{O(\mathfrak{L}_2^n)}$, $\overline{O(\mathfrak{L}_5^n)}$, $\overline{O(\mathfrak{L}_{11}^n)}$ and $\overline{\{O(\mathfrak{N}_{3}^{a})\}_{a \in\mathbb{C}}}$.

\begin{Lem}\label{norigidfam}
Any irreducible component in $\mathfrak{Leib}_4$ contains a non-nilpotent algebra.
\end{Lem}

\begin{proof} It is enough to prove that $\mathfrak{L}^n_2$, $\mathfrak{L}^n_5$, $\mathfrak{L}^n_{11}$ and $\mathfrak{N}_{3}^{a}$ belong to the closure of the union of orbits of non-nilpotent Lebniz algebras for any $a\in\mathbb{C}$. Let us prove this assertion case by case.
\begin{itemize}
\item The parametrized basis $E_1^t=te_1+te_4$, $E_2^t=t^2e_2+t^2e_4$, $E_3^t=t^3e_3+t^3e_4$, $E_4^t=-t^4e_3$ gives the degeneration $\mathfrak{L}_{40} \to \mathfrak{L}_{2}^n$.
\item The parametrized basis $E_1^t=te_1-te_2+e_3$, $E_2^t=t^2e_2+e_3$, $E_3^t=te_3-t^3e_4$, $E_4^t=t^4e_4$ with parametrized index $\epsilon(t)=-t$ give the assertion $\mathfrak{L}_{18}^* \to \mathfrak{L}_{5}^n$.
\item The parametrized basis $E_1^t=te_1-\frac{t}{2}e_2$, $E_2^t=\frac{t^2}{2}e_2+t^2e_3$, $E_3^t=t^3e_3-\frac{t^3}{4}e_4$, $E_4^t=\frac{t^4}{4}e_4$ with parametrized index $\epsilon(t)=\frac{t^2+1}{4}$ give the assertion $\mathfrak{L}_{15}^* \to \mathfrak{L}_{11}^n$.
\item The parametrized basis $ E_1^t= te_1+\frac{t}{2}e_4$, $E_2^t=ite_1-ie_3-\frac{it}{2}e_4$, $E_3^t=te_2$, $E_4^t= t^2e_4 $ gives the degeneration $\mathfrak{L}_{9}^{1+ia}  \to  \mathfrak{N}_{3}^{a}$ for any $a\in\mathbb{C}$.
\end{itemize}
Let us consider the degeneration $\mathfrak{L}_{40} \to  \mathfrak{L}_{2}^n$ to clarify our proof.  Writing nonzero products of $\mathfrak{L}_{40}$ in the basis  $E_i^t$, we get 
$$
E_1^tE_1^t= t^2e_2+t^2e_4 =E_2^t, \ E_2^tE_1^t=t^3e_3+t^3e_4=E_3^t, \ E_3^tE_1^t=t^4e_4=tE_3^t+E_4^t .$$
 It is easy to see now that for $t=0$ we obtain the multiplication table of $\mathfrak{L}_{2}^n.$  The remaining assertions can be considered in the same way.
\end{proof}

Though we have proved that $\overline{O(\mathfrak{L}_5^n)}$ and $\overline{O(\mathfrak{L}_{11}^n)}$ are not irreducible components of $\mathfrak{Leib}_4$, we have not found an algebra that degenerates either to $\mathfrak{L}_5^n$ or to $\mathfrak{L}_{11}^n$ in the proof of Lemma \ref{norigidfam}. In the proof of the next result we show that in fact there is no algebra that degenerates to $\mathfrak{L}_5^n$.

\begin{Th}
$\mathfrak{Leib}_4$ does not have Grunewald--O'Halloran Property.
\end{Th}
\begin{proof}
To prove the theorem, we will show that there is no four dimensional Leibniz algebra that degenerates to $\mathfrak{L}_5^n.$ Since $\mathfrak{L}_5^n$ is not a Lie algebra, only algebras from Table 1 can degenerate to it. Direct calculations show that the dimension of the Lie algebra of derivations is greater or equal to $3$ for all algebras from Table 1 except $\mathfrak{R}_1$ and $\mathfrak{L}_{44}$. Since $dim\,Der(\mathfrak{L}_5^n)=3$ (see, for example \cite{kppv}) it suffices to prove that $\mathfrak{R}_1\not\to\mathfrak{L}_5^n$ and $\mathfrak{L}_{44}\not\to\mathfrak{L}_5^n$.
\begin{itemize}
\item $\mathfrak{R}_1\not\to\mathfrak{L}_5^n$ follows from the fact that $Ann_L(\mathfrak{R}_1)>Ann_L(\mathfrak{L}_5^n)$.
\item To prove the assertion $\mathfrak{L}_{44}\not\to\mathfrak{L}_5^n$ let us consider the set
$$
\mathcal{R}=\left\{\mu\in\mathfrak{Leib}_4\left|\begin{array}{c}S_1S_3+S_4S_2=0,S_3S_2+S_4S_1\subset S_4,S_2S_2+S_3S_1\subset S_3, S_2S_1+S_1S_2\subset S_2,\\c_{12}^2+c_{21}^2=0,c_{31}^3=2c_{21}^2,c_{12}^3=c_{21}^3\end{array}\right.\right\}.
$$
It is not difficult to show that $\mathcal{R}$ is a closed subset of $\mathfrak{Leib}_4$ that is stable under the action of the subgroup of lower triangular matrices and contains the structure $\mathfrak{L}_{44}$. It is also not difficult to show that $\mathcal{R}\cap O(\mathfrak{L}_5^n)=\varnothing$.
\end{itemize}
\end{proof}

\subsection{The proof of the main theorem.} Now we are ready to prove Theorem \ref{mainTh}. As a first step, we are going to prove that the irreducible components of the variety of four dimensional Lie algebras remain irreducible in $\mathfrak{Leib}_4$. This fact follows from the next general lemma.

\begin{Lem}\label{trivAnn} Suppose that $A(*)=\{A(\alpha)\}_{\alpha\in T}$ is a set of $n$-dimensional non-Lie Leibniz algebras. If $B$ is an $n$-dimensional Lie algebra such that $Ann(B)=0$, then $A(*)\not\to B$.
\end{Lem}
\begin{Proof} Since $A(\alpha)$ is non-Lie, the ideal $I(\alpha)=\{xy+yx\mid x,y\in A(\alpha)\}$ is nonzero for any $\alpha\in T$. Since $I(\alpha)\subset Ann_L\big(A(\alpha)\big)$, we have $Ann_L\big(A(\alpha)\big)>0=Ann(B)=Ann_L(B)$ for any $\alpha\in T$, and hence $A(*)\not\to B$.
\end{Proof}

It easily follows from Lemma \ref{trivAnn} that if $\mathcal{C}$ is an irreducible component in the variety of $n$-dimensional Lie algebras containing an algebra with  zero annihilator, then $\mathcal{C}$ is an irreducible component in $\mathfrak{Leib}_n$. Indeed, suppose that $\mathcal{C}_0$ is an irreducible component of $\mathfrak{Leib}_n$ containing $\mathcal{C}$. Let $\mathcal{L}$ be a set of all Lie algebras in $\mathcal{C}_0$.  We have  $\mathcal{C}\not\subset\overline{\mathcal{C}_0\setminus\mathcal{L}}$ by Lemma \ref{trivAnn}. Since $\mathcal{C}_0=\overline{\mathcal{C}_0\setminus\mathcal{L}}\cup\mathcal{L}$ and $\mathcal{C}_0$ is irreducible, we have $\mathcal{C}_0=\mathcal{L}$, and hence $\mathcal{C}_0=\mathcal{C}$.

\begin{corollary}\label{compLie} $\overline{O(\mathfrak{sl}_2\oplus \mathbb C)}$, $\overline{O(\mathfrak{R}_2)}$, $\overline{\{O(g_5(a))\}_{a \in\mathbb{C}}}$ and $\overline{\{O(g_4(a,b))\}_{a,b \in\mathbb{C}}}$ are irreducible components of $\mathfrak{Leib}_4$.
\end{corollary}
\begin{Proof} The structure $\mathfrak{sl}_2\oplus \mathbb C$ is rigid in $\mathfrak{Leib}_4$ as a unique non-solvable structure in the variety. The remaining part of the corollary follows from the fact that $Ann(\mathfrak{R}_2)=0$, $Ann\big(g_5(a)\big)=0$ for $a\not=-1$, and $Ann\big(g_4(a,b)\big)=0$ for all $a,b\in\mathbb{C}^*$.
\end{Proof}

\begin{Proof}[Proof of Theorem \ref{mainTh}.] Let $\mathcal{W}$ be the union of closed sets listed in the theorem. The assertions proved in Table 2 and the classification given in Table 1 show that all four dimensional non-nilpotent non-Lie Leibniz algebras belong to $\mathcal{W}$. All Lie algebras belong to $\mathcal{W}$ by the results of the paper \cite{BC99}.
Then it follows from Lemma \ref{norigidfam} that all nilpotent Leibniz algebras belong to $\mathcal{W}$ too.

Thus, it remains to show that for any two different sets $\mathcal{X}$ and $\mathcal{Y}$ listed in the theorem $\mathcal{X}\not\subset \mathcal{Y}$. It follows from Corollary \ref{compLie} and the fact that the set of four dimensional Lie algebras is a closed subset of $\mathfrak{Leib}_4$ that the required assertion is true if $\mathcal{X}$ or $\mathcal{Y}$ is formed by Lie algebras.

Let us start with the algebras $\mathfrak{R}_1$ and $\mathfrak{R}_3$. These are the only two algebras from Table 1 that have two dimensional nilpotent radical. Since all other algebras have a three dimensional nilpotent radical, $\mathfrak{R}_1$ and $\mathfrak{R}_3$ do not belong to the closure of orbits of all the remaining structures of $\mathfrak{Leib}_4$. Since $Ann_L(\mathfrak{R}_1)>Ann_L(\mathfrak{R}_3)$ and $(\mathfrak{R}_3)^{(+2)}<(\mathfrak{R}_1)^{(+2)}$, we have $\mathfrak{R}_1\not\to\mathfrak{R}_3$ and $\mathfrak{R}_3\not\to\mathfrak{R}_1$. Thus, $\overline{O(\mathfrak{R}_1)}$ and $\overline{O(\mathfrak{R}_3)}$ are irreducible components. Since $dim\,(\mathfrak{R}_1)^2=dim\,(\mathfrak{R}_3)^2=2$, these two components do not contain $\{\mathfrak{L}_4^a \}_{a \in\mathbb{C}}$,
$\{\mathfrak{L}_8^a \}_{a \in\mathbb{C}}$,
$\{\mathfrak{L}_9^a \}_{a \in\mathbb{C}}$,
$\{\mathfrak{L}_{10}^a \}_{a \in\mathbb{C}}$,
$\{\mathfrak{L}_{21}^{a,b} \}_{a,b \in\mathbb{C}}$,
$\{\mathfrak{L}_{22}^{a,b} \}_{a,b \in\mathbb{C}}$,
$\{\mathfrak{L}_{23}^{a,b} \}_{a,b \in\mathbb{C}}$,
$\mathfrak{L}_2$, and
$\mathfrak{L}_{44}$. Since $Ann_L(\mathfrak{R}_1)>Ann_L(\mathfrak{L}_{15}^a)=Ann_L(\mathfrak{L}_{18}^a)$ for any $a\not=0$, we have $\mathfrak{R}_1\not\to\mathfrak{L}_{15}^a,\mathfrak{L}_{18}^a$ for $a\not=0$. Note that $\mathfrak{R}_3$ contains a three dimensional anticommutative subalgebra $D=\langle e_2,e_3,e_4\rangle$ such that $dim\,(\mathfrak{R}_3D)=1$. Since both $\mathfrak{L}_{15}^a$ and $\mathfrak{L}_{18}^a$ do not have such a subalgebra for any $a\not=0$, we have $\mathfrak{R}_3\not\to\mathfrak{L}_{15}^a,\mathfrak{L}_{18}^a$ for $a\not=0$.

All the remaining algebras are solvable non-nilpotent Leibniz algebras with a three dimensional nilpotent radical. Moreover, one can check that each of them is represented by a structure $\mu$ such that $\langle e_2,e_3,e_4\rangle$ is the nilpotent radical and, moreover, the structure constants $c_{ij}^k$ ($1\le i,j,k\le 4$) satisfy the conditions $c_{ij}^k=0$ if $i,j\ge 2$ and $k\le \max(i,j)$ and $c_{1i}^j=c_{i1}^j=0$ for any $2\le i,j\le 4$ such that $j<i$. During this proof we will call a structure with three dimensional nilpotent radical that satisfies the described conditions a {\it standard structure}.
Let us put in correspondence to a standard structure $\mu$ the $6$-tuple $S_{\mu}=(c_{21}^2,c_{12}^2,c_{31}^3,c_{13}^3,c_{41}^4,c_{14}^4)\in\mathbb{C}^6$. It is not difficult to show that if $S_{\mu}=(a_1,b_1,a_2,b_2,a_3,b_3)$ and $\lambda\in O(\mu)$ is a standard structure, then there is some permutation $\sigma:\{1,2,3\}\rightarrow\{1,2,3\}$ and some $c\in\mathbb{C}^*$ such that $S_{\lambda}=(ca_{\sigma(1)},cb_{\sigma(1)},ca_{\sigma(2)},cb_{\sigma(2)},ca_{\sigma(3)},cb_{\sigma(3)})$.
Indeed, suppose that $\lambda$ and $\mu$ are standard structures and $\lambda=g^{-1}*\mu$ for some $g\in GL(V)$. Let us denote by $U$ the subspace $\langle e_2,e_3,e_4\rangle$ of $V$. Since $U$ is the nilpotent radical for $\mu$ and $\lambda$ at the same time, we have $g(U)=U$.
We have also $g(e_1)=ce_1+u$ for some $c\in\mathbb{C}^*$, and $u\in U$. Let $L,R:U\rightarrow U$ denote the operators of the left and right multiplications by $e_1$ with respect to $\lambda$ correspondingly.
Note that at the same time $L$ and $R$ are the operators of the left and right multiplications by $g(e_1)$ with respect to $\mu$ conjugated by $g$.
Then it follows from the definition of a standard structure that, for any $s,t\in\mathbb{C}$, on one hand the characteristic polynomial $\chi_{sR+tL}(x)\in\mathbb{C}[x]$ of $sR+tL$ is $\prod\limits_{i=1}^3\big(c(sa_i+tb_i)-x\big)$ and, on the other hand, the same polynomial is $\prod\limits_{i=1}^3(sa_i'+tb_i'-x)$, where $S_{\lambda}=(a_1',b_1',a_2',b_2',a_3',b_3')$. The comparing of the roots of the obtained polynomials gives the required assertion.

Suppose now that $\{\mu_{\alpha}\}_{\alpha\in T}$ is a set of standard structures, $S_{\mu_{\alpha}}=(a_{1,\alpha},b_{1,\alpha},a_{2,\alpha},b_{2,\alpha},a_{3,\alpha},b_{3,\alpha})$, and the linear homogeneous polynomials $f_1,\dots,f_l\in\mathbb{C}[x_1,x_2,x_3,x_4,x_5,x_6]$ are such that $f_r(a_{1,\alpha},b_{1,\alpha},a_{2,\alpha},b_{2,\alpha},a_{3,\alpha},b_{3,\alpha})=0$ for all $\alpha\in T$ and $1\le r\le l$. If $\lambda\in\overline{\{O(\mu(\alpha))\}_{\alpha\in T}}$ is a standard structure with $S_{\lambda}=(a_1,b_1,a_2,b_2,a_3,b_3)$, then there is some permutation $\sigma:\{1,2,3\}\rightarrow\{1,2,3\}$ and some $c\in\mathbb{C}^*$ such that $f_r(ca_{\sigma(1)},cb_{\sigma(1)},ca_{\sigma(2)},cb_{\sigma(2)},ca_{\sigma(3)},cb_{\sigma(3)})=0$ for all $1\le r\le l$. Indeed, the conditions from the definition of a standard structure with the conditions $f_r(c_{21}^2,c_{12}^2,c_{31}^3,c_{13}^3,c_{41}^4,c_{14}^4)=0$ ($1\le r\le l$) determine a closed subset of $\mathfrak{Leib}_4$ that is invariant under the action of lower triangular matrices. Thus, it follows from Lemma \ref{gmain} that $\lambda$ belongs to an orbit of some standard structure satisfying the equalities $f_r(c_{21}^2,c_{12}^2,c_{31}^3,c_{13}^3,c_{41}^4,c_{14}^4)=0$ ($1\le r\le l$).

Computation of $6$-tuples for the structures under consideration gives the following results:
$$\begin{array} {lll}
S_{\mathfrak{L}_2}=(1,-1,0,-1,1,0),         & S_{\mathfrak{L}_4^a}=(1,a,a+1,-1,0,0),        & S_{\mathfrak{L}_8^a}=(1,a,a+1,-1,-a,0),\\
S_{\mathfrak{L}_9^a}=(1,a,2,-1,-a,0),       & S_{\mathfrak{L}_{10}^a}=(1,a,2,-1,0,0),       & S_{\mathfrak{L}_{15}^a}=(0,1,0,0,-1,0),\\
S_{\mathfrak{L}_{18}^a}=(0,1,0,0,0,0),      & S_{\mathfrak{L}_{21}^{a,b}}=(1,a,b,-1,-a,0),  & S_{ \mathfrak{L}_{22}^{a,b}}=(1,a,b,-1,0,0),\\
S_{\mathfrak{L}_{23}^{a,b}}=(1,a,b,0,0,0),  & S_{\mathfrak{L}_{44}}=(1,2,3,-1,0,0).
\end{array}$$
Almost all the required assertions follow now from our $6$-tuple argument. Also the assertions of the form $A(*)\not\to B$, where $A(*)\in\{\mathfrak{L}_{21}^{*,*},\mathfrak{L}_{22}^{*,*},\mathfrak{L}_{23}^{*,*}\}$ and $B\in \{\mathfrak{L}_2,\mathfrak{L}_4^a,\mathfrak{L}_8^a,\mathfrak{L}_9^a,\mathfrak{L}_{10}^a,\mathfrak{L}_{15}^a,\mathfrak{L}_{18}^a,\mathfrak{L}_{44}\}_{a\in\mathbb{C}}$ follow from the fact that $A(a,b)$ has three dimensional trivial subalgebra for any $a,b\in\mathbb{C}$ and $B$ does not have three dimensional trivial subalgebra.
Let us consider the remaining assertions case by case.
\begin{itemize}
\item $\mathfrak{L}_4^*\not\to \mathfrak{L}_{44}$. The required assertion follows from the fact that $Der(\mathfrak{L}_{44})>Der(\mathfrak{L}_4^a)$ for any $a\in\mathbb{C}$.
\item To prove the assertions $\mathfrak{L}_9^*\not\to \mathfrak{L}_{15}^a$ and $\mathfrak{L}_{10}^*\not\to \mathfrak{L}_{18}^a$ let us consider the set
$$
\mathcal{R}=\left\{\mu\in\mathfrak{Leib}_4\left|\begin{array}{c}S_1S_1\subset S_2, S_2S_2\subset S_4,S_3S_1+S_1S_3\subset S_3,S_4S_1\subset S_4, S_1S_4+S_4S_2=0,\\c_{13}^3+c_{31}^3=0,c_{41}^4=2c_{31}^3,c_{13}^4=c_{31}^4,c_{23}^4=c_{32}^4\end{array}\right.\right\}.
$$
It is not difficult to show that $\mathcal{R}$ is a closed subset of $\mathfrak{Leib}_4$ that is stable under the action of the subgroup of lower triangular matrices and contains the structures $\mathfrak{L}_9^b$ and $\mathfrak{L}_{10}^b$ for any $b\in\mathbb{C}$. To see this it is enough to consider the basis $e_1$, $e_3$, $e_2$, $e_4$.
It is also not difficult to show that $\mathcal{R}\cap O(\mathfrak{L}_{15}^a)=\mathcal{R}\cap O(\mathfrak{L}_{18}^a)=\varnothing$ for any $a\in\mathbb{C}$.
\end{itemize}

\end{Proof}

\section{Appendix: Tables}


Table 1. {\it Four dimensional non-nilpotent non-Lie Leibniz algebras.}

\begin{longtable}{l|llllllllll} 
 
$\mathfrak{R}_1 $& 
$e_3e_1 = e_3$ &$e_4e_2=e_4$& \\ 
\hline

$\mathfrak{R}_3$ & 
$e_2e_4=-e_4$ & $e_3e_1 = e_3$ & $e_4e_2=e_4$\\ 
\hline 

$\mathfrak{L}_2$ & 
$e_1e_1=e_4$ & $e_1e_2=-e_2$ & $e_1e_3=e_3$ & $e_2e_1=e_2$ \\ & $e_2e_3=e_4$ &  $e_3e_1=-e_3$ &  $e_3e_2=-e_4$\\ 
\hline

$\mathfrak{L}_4^{a}$ & 
$e_1e_2=-e_2$ & $e_2e_1=e_2$ & $e_3e_1= a  e_3$  \\ & $e_3e_2=e_4$  & $e_4e_1=(1+a)e_4$\\ 
\hline

$\mathfrak{L}_5$  & 
$e_1e_1=e_4$ & $e_1e_2=-e_2$ &  $e_2e_1=e_2$ & $e_3e_1= - e_3$ &   $e_3e_2=e_4$ \\ 
\hline

$\mathfrak{L}_6$ & 
$e_1e_1=-e_3$ & $e_1e_2=-e_2$ & $e_2e_1=e_2+e_4$&   $e_3e_2=e_4$ &  $e_4e_1=e_4$\\ 
\hline

$\mathfrak{L}_7$ & 
$e_3e_1=e_3$& $e_3e_2=e_4$ &  $e_4e_1=e_4$ \\ 
\hline

$\mathfrak{L}_8^{a \neq-1,0}$ & 
$e_1e_2=-e_2$ &  $e_1e_3=-a e_3$ & $e_2e_1=e_2$ & $e_2e_3=a e_4$ \\&  $e_3e_1=a  e_3$  & $e_3e_2=e_4$ &  $e_4e_1=(a+1)e_4$ \\ 
\hline

$\mathfrak{L}_9^{a}$ & 
$e_1e_2=-e_2$ & $e_1e_3=-a e_3$ &  $e_2e_1=e_2$ & $e_2e_2=e_4$ \\&  $e_3e_1=a e_3$ &  $e_4e_1=2e_4$\\ 
\hline

$\mathfrak{L}_{10}^{a}$ & 
$e_1e_2=-e_2$ & $e_2e_1=e_2$ & $e_2e_2=e_4$ &  $e_3e_1=a e_3$ &  $e_4e_1=2e_4$\\ 
\hline

$\mathfrak{L}_{11}$ & 
$e_1e_1=e_3$ & $e_1e_2=-e_2$ & $e_2e_1=e_2$ & $e_2e_2=e_4$ &   $e_4e_1=2e_4$\\ 
\hline

$
\mathfrak{L}_{12}$ & 
$e_1e_2=-e_2$ &  $e_2e_1=e_2$ &   $e_2e_2=e_4$ & $e_3e_1=2e_3+e_4$ & $e_4e_1=2e_4$ \\ 
\hline

$\mathfrak{L}_{13}$ & 
$e_1e_2=-e_2-e_3$ & $e_1e_3=-e_3$ & $e_2e_1=e_2+e_3$ & $e_2e_2=e_4$\\ & $e_3e_1=e_3$ &  $e_4e_1=2e_4$  \\ 
\hline



$\mathfrak{L}_{14}$ & 
$e_1e_3=-e_3$ & $e_2e_2=e_4$ & $e_3e_1=e_3$ \\ 
\hline

$\mathfrak{L}_{15}^{a}$ & 
$e_1e_1=a e_4$ & $e_1e_2=e_4$ & $e_1e_3=-e_3$ &  $e_2e_2=e_4$ &   $e_3e_1=e_3$\\ 
\hline

$\mathfrak{L}_{16}$ & 
$e_1e_1=-2e_4$ &  $e_1e_3=-e_3$ &   $e_2e_2=e_4$ & $e_3e_1=e_3$     \\ 
\hline

$\mathfrak{L}_{17}$ & 
$e_2e_2=e_4$ & $e_3e_1=e_3$\\ 
\hline

$\mathfrak{L}_{18}^{a}$ & 
$e_1e_1=a  e_4 $&$ e_1e_2=e_4$& $e_2e_2=e_4$ & $e_3e_1=e_3 $      \\ 
\hline

$\mathfrak{L}_{19}$ & 
$ e_1e_2=e_4 $ & $e_2e_1=e_4 $& $e_2e_2=e_4$ & $e_3e_1=e_3$    \\ 
\hline

$\mathfrak{L}_{21}^{a, b \neq 0}$ & 
$e_1e_2=-e_2$ & $e_1e_3=- a  e_3$ & 
$e_2e_1=e_2$ & $e_3e_1= a  e_3$ &  $e_4e_1= b  e_4$ \\ 
\hline

$\mathfrak{L}_{22}^{a  \neq 0, b \neq 0}$ & 
$e_1e_2=-e_2$& $e_2e_1=e_2$ & $e_3e_1= a  e_3$ &  $e_4e_1= b  e_4$      \\ 
\hline

$\mathfrak{L}_{23}^{a,b}$ & 
$e_2e_1=e_2$ & $e_3e_1= a e_3$ &  $e_4e_1= b e_4$ \\ 
\hline

$\mathfrak{L}_{24}^{a}$ & 
$e_1e_1=e_4$ & $e_1e_2= - e_2$  &$e_1e_3= -a e_3$ & $e_2e_1=e_2$ & $e_3e_1= a e_3$ \\ 
\hline

$\mathfrak{L}_{25}^{a\neq 0}$ & 
$e_1e_1 =e_4$&  $e_1e_2=-e_2$&  $e_2e_1=e_2$ & $e_3e_1= a e_3$       \\ 
\hline

$\mathfrak{L}_{26}^{a}$ & 
$e_1e_1=e_4$& $e_2e_1=e_2$ & $e_3e_1= a e_3$ \\ 
\hline

$\mathfrak{L}_{27}$ & 
$e_1e_2= -e_2$ &  $e_1e_3= e_4$ & $e_2e_1=e_2$ \\ 
\hline

$\mathfrak{L}_{28}$ & 
$e_1e_3= e_4$ & $e_2e_1=e_2$ \\ 
\hline

$\mathfrak{L}_{29}^{a\neq0}$ & 
$e_1e_2=-e_2-e_3$ & $e_1e_3=- e_3$ & $e_2e_1=e_2+e_3$ & $e_3e_1= e_3$ &  $e_4e_1= a e_4$ \\ 
\hline

$\mathfrak{L}_{30}$ & 
$e_1e_1 =e_4$& $e_1e_2=-e_2-e_3$ & $ e_1e_3=- e_3$ &  $e_2e_1=e_2+e_3$ & $e_3e_1= e_3$ \\ 
\hline

$\mathfrak{L}_{32}^{a}$ & 
$e_2e_1=e_2+e_3$ & $e_3e_1= e_3$ &  $e_4e_1= a e_4$ \\ 
\hline

$\mathfrak{L}_{33}$& 
$e_1e_1=  e_4$& $e_2e_1=e_2+e_3$ & $e_3e_1= e_3$\\ 
\hline

$\mathfrak{L}_{34}^{a \neq 0}$ & 
$e_1e_4 =- a e_4 $ & $e_2e_1=e_2+e_3$ & $e_3e_1= e_3$ &  $e_4e_1= a e_4$    \\ 
\hline

$\mathfrak{L}_{35}^{a\neq-1}$ & 
$e_1e_2=a e_3$ & $e_1e_4=- e_4$ & $e_2e_1=e_3$ & $e_4e_1= e_4$\\ 
\hline

$\mathfrak{L}_{36}$ & 
$e_1e_1 =e_3$&  $e_1e_2=-e_3$ & $e_1e_4=- e_4$ & $e_2e_1=e_3$ & $e_4e_1= e_4$ \\ 
\hline

$\mathfrak{L}_{37}$& 
$e_1e_1=e_2$ &  $e_1e_4= - e_4$ & $e_2e_1=e_3$ & $e_4e_1=  e_4$\\ 

\hline

$\mathfrak{L}_{38}^{a}$ & 
$e_1e_2=ae_3$ & $e_2e_1=e_3$ & $e_4e_1=  e_4$ \\ 
\hline

$\mathfrak{L}_{39}$ & 
$e_1e_1=e_3$ & $e_1e_2= - e_3$ & $e_2e_1=e_3$ & $e_4e_1= e_4$ \\ 
\hline

$\mathfrak{L}_{40}$ & 
$e_1e_1=  e_2$& $e_2e_1=e_3$ & $e_4e_1= e_4$ \\ 
\hline

$\mathfrak{L}_{41}$ & 
$e_2e_1=e_2+e_3$ & $e_3e_1= e_3+e_4$ &  $e_4e_1=  e_4$ \\ 
\hline

$\mathfrak{L}_{44}$ & 
$e_1e_2 =- e_2$ & $e_2e_1=  e_2$ & $e_2e_2=e_3$ 
&   $e_3e_1= 2e_3$ \\ 
& $e_3e_2= e_4$&   $e_4e_1=3e_4$ 
\\ 
\hline

\end{longtable}

\normalsize 

\

Table 2.   
{\it Orbit closures for some families and degenerations of four dimensional Leibniz algebras.}

 {\tiny$$\begin{array}{|rcl|llll|l|}
\hline

\multicolumn{3}{|l|}{\mbox{Assertions}} &  \mbox{Parametrized bases} &&& & \mbox{Parametrized indices} \\
\hline
\hline

{\red
\mathfrak{L}_{8}^* } &\to& \mathfrak{L}_{2}&

 E_1^t= e_1 + (t-1) e_4,&
 E_2^t= t e_2,&
 E_3^t= e_3,&
 E_4^t=(t-1) t e_4
& \epsilon(t)=t-1 \\

\hline
\mathfrak{L}_{4}^*  &\to& \mathfrak{L}_{5}      & E_1^t=e_1+e_4, & E_2^t=e_2, & E_3^t=te_3, & E_4^t=te_4 & \epsilon(t)=t-1  \\

\hline
\mathfrak{L}_{4}^*  &\to& \mathfrak{L}_{6}      & E_1^t=e_1-e_3, & E_2^t=e_2+e_4, & E_3^t=te_3, & E_4^t=te_4 & \epsilon(t)=t \\

\hline
\mathfrak{R}_1      &\to& \mathfrak{L}_{7}      & E_1^t=e_1+e_2, & E_2^t=te_2, & E_3^t=e_3+e_4, & E_4^t=te_4 & \\

\hline
\mathfrak{L}_{10}^* &\to& \mathfrak{L}_{11}     & E_1^t=e_1+e_3, & E_2^t=e_2, & E_3^t=te_3, & E_4^t=e_4 & \epsilon(t)=t \\

\hline
\mathfrak{L}_{10}^* &\to& \mathfrak{L}_{12}     & E_1^t=e_1, & E_2^t=te_2, & E_3^t=e_3+e_4, & E_4^t=t^2e_4 & \epsilon(t)=2-t^2 \\

\hline
\mathfrak{L}_{9}^*  &\to& \mathfrak{L}_{13}     & E_1^t=e_1, & E_2^t=e_2+e_3, & E_3^t=te_3, & E_4^t=e_4 & \epsilon(t)=t+1 \\

\hline
\mathfrak{L}_{16}   &\to& \mathfrak{L}_{14}     & E_1^t=e_1, & E_2^t=\frac{1}{t}e_2, & E_3^t=e_3, & E_4^t=\frac{1}{t^2}e_4 & \\

\hline
\mathfrak{L}_{15}^* &\to& \mathfrak{L}_{16}     & E_1^t=e_1, & E_2^t=\frac{1}{t}e_2, & E_3^t=e_3, & E_4^t=\frac{1}{t^2}e_4 & \epsilon(t)=-\frac{2}{t^2} \\

\hline
\mathfrak{R}_1      &\to& \mathfrak{L}_{17}     & E_1^t=e_1, & E_2^t=te_2+e_4, & E_3^t=e_3, & E_4^t=te_4 & \\

\hline
\mathfrak{L}_{18}^* &\to& \mathfrak{L}_{19}     & E_1^t=e_1+\frac{1}{t}e_2, & E_2^t=\frac{1}{t}e_2, & E_3^t=e_3, & E_4^t=\frac{1}{t^2}e_4 & \epsilon(t)=-\frac{1}{t^2} \\

\hline
\mathfrak{L}_{21}^{*,*} &\to& \mathfrak{L}_{24}^{a} & E_1^t=e_1+e_4, & E_2^t=e_2, & E_3^t=e_3, & E_4^t=te_4 & \epsilon(t)=(a,t) \\

\hline
\mathfrak{L}_{22}^{*,*} &\to& \mathfrak{L}_{25}^{a} & E_1^t= e_1+e_4, & E_2^t=e_2, & E_3^t=e_3, & E_4^t=te_4 & \epsilon(t)=(a,t) \\

\hline
\mathfrak{L}_{23}^{*,*} &\to& \mathfrak{L}_{26}^{a} & E_1^t= e_1+e_4, & E_2^t=e_2, & E_3^t=e_3, & E_4^t=te_4 & \epsilon(t)=(a,t) \\

\hline
\mathfrak{R}_3      &\to& \mathfrak{L}_{27}    & E_1^t=e_2+e_3, & E_2^t=e_4, & E_3^t=te_1, & E_4^t=te_3 & \\

\hline
\mathfrak{R}_1      &\to& \mathfrak{L}_{28}    & E_1^t=e_1+e_4, & E_2^t=e_3, & E_3^t=te_2, & E_4^t=te_4 & \\

\hline
\mathfrak{L}_{21}^{*,*} &\to& \mathfrak{L}_{29}^{a} & E_1^t=e_1, & E_2^t=e_2+e_3, & E_3^t=te_3, & E_4^t=e_4 & \epsilon(t)=(t+1,a) \\

\hline
\mathfrak{L}_{29}^{*}   &\to& \mathfrak{L}_{30} & E_1^t=e_1+e_4, & E_2^t=e_2, & E_3^t=e_3, & E_4^t=te_4 & \epsilon(t)=t \\

\hline
\mathfrak{L}_{23}^{*,*} &\to& \mathfrak{L}_{32}^{a} & E_1^t=e_1, & E_2^t=e_2+e_3, & E_3^t=te_3, & E_4^t=e_4 & \epsilon(t)=(t+1,a) \\

\hline
\mathfrak{L}_{26}^{*}   &\to& \mathfrak{L}_{33} & E_1^t=e_1, & E_2^t=e_2+e_3, & E_3^t=te_3, & E_4^t=e_4 & \epsilon(t)=t+1 \\

\hline
\mathfrak{L}_{22}^{*,*} &\to& \mathfrak{L}_{34}^{a} & E_1^t=a e_1, & E_2^t=e_3+e_4, & E_3^t=te_4, & E_4^t=e_2 & \epsilon(t)=\left(\frac{1}{a},\frac{t+1}{a}\right) \\

\hline
\mathfrak{L}_{15}^{\frac{-a}{(1-a)^2}} &\to&  \mathfrak{L}_{35}^{a\not=1}    & E_1^t=e_1+\frac{1}{a-1}e_2, & E_2^t=t(a-1)e_2, & E_3^t=te_4, & E_4^t=e_3 &\\

\hline
\mathfrak{L}_{24}^{*} &\to& \mathfrak{L}_{36} & E_1^t=te_1, & E_2^t=e_2-te_4, & E_3^t=t^2e_4, & E_4^t=e_3 & \epsilon(t)=\frac{1}{t} \\

\hline
\mathfrak{L}_{34}^{*} &\to& \mathfrak{L}_{37} & E_1^t=te_1+e_2, & E_2^t=t(e_2+e_3), & E_3^t=t^2e_3, & E_4^t=e_4 & \epsilon(t)=\frac{1}{t} \\

\hline
\mathfrak{L}_{18}^{\frac{-a}{(1-a)^2}} & \to &  \mathfrak{L}_{38}^{a\not=1}    &  E_1^t=e_1+\frac{1}{a-1}e_2, & E_2^t=t(a-1)e_2, & E_3^t=te_4, & E_4^t=e_3 & \\

\hline
\mathfrak{L}_{22}^{*,*} &\to& \mathfrak{L}_{39} & E_1^t=te_1+\frac{1}{t}e_3, & E_2^t=e_2-e_3, & E_3^t=te_3, & E_4^t=e_4 & \epsilon(t)=\left(t,\frac{1}{t}\right) \\

\hline
\mathfrak{L}_{23}^{*,*} &\to& \mathfrak{L}_{40} & E_1^t=te_1+e_2+e_3, & E_2^t=te_2+2te_3, & E_3^t=2t^2e_3, & E_4^t=e_4 & \epsilon(t)=\left(2,\frac{1}{t}\right) \\

\hline
\mathfrak{L}_{23}^{*,*} &\to& \mathfrak{L}_{41} & E_1^t=e_1, & E_2^t=e_2+2e_3+e_4, & E_3^t=2t(e_3+e_4), & E_4^t=2t^2e_4 & \epsilon(t)=(t+1,2t+1) \\

\hline

\end{array}$$

}

\normalsize


\begin{thebibliography}{99}



\bibitem{alb06} 
Albeverio S., Omirov B., Rakhimov I., 
Classification of $4$-dimensional nilpotent complex Leibniz algebras, 
Extracta Mathematicae, 21 (2006), 3, 197--210. 


\bibitem{barnes}
Barnes D., 
On Levi’s theorem for Leibniz algebras, 
Bulletin of the Australian Mathematical Society, 86 (2012), 184--185.

\bibitem{BB14} 
Benes T., Burde D., 
Classification of orbit closures in the variety of three-dimensional Novikov algebras, 
Journal of Algebra and Its Applications, 13 (2014), 2, 1350081, 33 pp.

\bibitem{BC99} 
Burde D., Steinhoff C., 
Classification of orbit closures of $4$--dimensional complex Lie algebras,   
Journal of Algebra, 214 (1999), 2, 729--739.

\bibitem{Abror13} 
Ca\~{n}ete E., Khudoyberdiyev A.,
The classification of $4$-dimensional Leibniz algebras,
Linear Algebra and its Applications, 439 (2013), 1, 273-288.


\bibitem{leib4}
Calderon A., Camacho L., Omirov B., 
Leibniz algebras of Heisenberg type,  
Journal of Algebra, 452 (2016), 427--447.

\bibitem{CKLO13}
Casas J., Khudoyberdiyev A., Ladra M., Omirov B., On the degenerations of solvable Leibniz algebras, 
Linear Algebra and its Applications,  439 (2013),  2, 472--487.

\bibitem{leib2}
Dherin B., Wagemann F., 
Deformation quantization of Leibniz algebras, 
Advances in Mathematics, 270 (2015), 21--48. 

\bibitem{GRH}
Grunewald F.,  O'Halloran J., 
Varieties of nilpotent Lie algebras of dimension less than six, 
Journal of Algebra, 112 (1988), 315--325.

\bibitem{GRH2}
Grunewald F., O'Halloran J., 
A Characterization of Orbit Closure and Applications, 
Journal of Algebra, 116 (1988), 163--175.

\bibitem{GRH3}
Grunewald F., O'Halloran J., 
Deformations of Lie algebras, 
Journal of Algebra, 162 (1993), 1, 210--224. 

\bibitem{ht14}
Herrera-Granada J.F., Tirao P.,
Filiform Lie algebras of dimension $8$ as degenerations,
Journal of Algebra and Its Applications, 13 (2014), 1350144.
  
\bibitem{ht16}
Herrera-Granada J.F., Tirao P.,
The Grunewald--O'Halloran Conjecture for Nilpotent Lie Algebras of Rank $\geq1,$
Communications in Algebra, 44 (2016), 5,  2180--2192.
   
\bibitem{kv17}
Kaygorodov I.,   Volkov Yu., 
The variety of $2$-dimensional algebras over an algebraically closed field, 
Canadian Journal of Mathematics, 71 (2019), 4, 819--842.

 
\bibitem{kpv16}
Kaygorodov I.,  Popov Yu., Volkov Yu., 
Degenerations of binary Lie and nilpotent Malcev algebras, 
 Communications in Algebra,
  46 (2018), 11, 4929--4941. 

 
\bibitem{kppv}
Kaygorodov I.,  Popov Yu., Pozhidaev A., Volkov Yu., 
Degenerations of Zinbiel and  nilpotent Leibniz algebras, 
Linear and Multilinear algebra, 
66 (2018), 4, 704--716.

\bibitem{KE14}  Kashuba I., Martin M., Deformations of Jordan algebras of dimension four, Journal of Algebra, 399 (2014), 277--289.


\bibitem{lodaypir}
Loday J.-L., Pirashvili T., 
Universal enveloping algebras of Leibniz algebras and (co)homology, 
Mathematische Annalen, 296 (1993), 1, 139--158. 

\bibitem{M79} 
Mazzola G., 
The algebraic and geometric classification of associative algebras of dimension five, 
Manuscripta Mathematica, 27 (1979), 81--101. 

\bibitem{yau}
Mostovoy J., 
A comment on the integration of Leibniz algebras, 
Communications in Algebra, 41 (2013), 1, 185--194. 


\bibitem{ort13}
Omirov B., Rakhimov I., Turdibaev R.,
On Description of Leibniz Algebras Corresponding to $sl_2$,
Algebras and Representation Theory, 16 (2013), 5, 1507--1519.


\bibitem{ra12}
Rakhimov I., Mohd Atan K.,
On Contractions and Invariants of Leibniz Algebras,
Bulletin of the Malaysian Mathematical Sciences Society,  35 (2012), 557--565.

\bibitem{S90}
Seeley C., 
Degenerations of $6$-dimensional nilpotent Lie algebras over $\mathbb{C}$, 
Communications in Algebra, 18 (1990), 3493--3505.

\bibitem{V70}
Vergne M., 
Cohomologie des algebres de Lie nilpotentes, 
Bulletin de la Société Mathématique de France, 98 (1970), 81--116.



\end{thebibliography}
\end{document}